\makeatletter\let\ifGm@compatii\relax\makeatother

\documentclass[10pt]{amsart}
\usepackage{latexsym}
\usepackage{amssymb}
\usepackage{amscd}
\usepackage{epsf}
\usepackage{amsmath,amssymb,graphicx,amsthm,mathrsfs,verbatim}
\usepackage{pstricks}

\begin{document}

%Style Declaration
\theoremstyle{plain}
\newtheorem{Thm}{Theorem}[section]
\newtheorem{TitleThm}[Thm]{}
\newtheorem{Corollary}[Thm]{Corollary}
\newtheorem{Proposition}[Thm]{Proposition}
\newtheorem{Lemma}[Thm]{Lemma}
\newtheorem{Conjecture}[Thm]{Conjecture}
\theoremstyle{definition}
\newtheorem{Definition}[Thm]{Definition}
\theoremstyle{definition}
\newtheorem{Example}[Thm]{Example}
\newtheorem{TitleExample}[Thm]{}
\newtheorem{Remark}[Thm]{Remark}
\newtheorem{SimpRemark}{Remark}
\renewcommand{\theSimpRemark}{}

%renewcommand{\theequation}{\thesection.\arabic{equation}}
\numberwithin{equation}{section}

\newcommand{\C}{{\mathbb C}}
\newcommand{\Q}{{\mathbb Q}}
\newcommand{\R}{{\mathbb R}}
\newcommand{\Z}{{\mathbb Z}}
\newcommand{\mbS}{{\mathbb S}}
\newcommand{\mbU}{{\mathbb U}}
\newcommand{\mbO}{{\mathbb O}}
\newcommand{\mbG}{{\mathbb G}}
\newcommand{\mbH}{{\mathbb H}}

\newcommand{\flushpar}{\par \noindent}

\newcommand{\proj}{{\rm proj}}
\newcommand{\coker}{{\rm coker}\,}
\newcommand{\Sol}{{\rm Sol}}
\newcommand{\supp}{{\rm supp}\,}
\newcommand{\codim}{{\operatorname{codim}}}
\newcommand{\sing}{{\operatorname{sing}}}
\newcommand{\Tor}{{\operatorname{Tor}}}
\newcommand{\Hom}{{\operatorname{Hom}}}
\newcommand{\wt}{{\operatorname{wt}}}
\newcommand{\graph}{{\operatorname{graph}}}
\newcommand{\rk}{{\operatorname{rk}}}
\newcommand{\dlog}{{\operatorname{Derlog}}}
\newcommand{\Olog}[2]{\Omega^{#1}(\text{log}#2)}
\newcommand{\produnion}{\cup \negmedspace \negmedspace 
\negmedspace\negmedspace {\scriptstyle \times}}
\newcommand{\pd}[2]{\dfrac{\partial#1}{\partial#2}}

\def \ba {\mathbf {a}}
\def \bb {\mathbf {b}}
\def \bc {\mathbf {c}}
\def \bd {\mathbf {d}}
\def \bone {\boldsymbol {1}}
\def \bg {\mathbf {g}}
\def \bG {\mathbf {G}}
\def \bh {\mathbf {h}}
\def \bi {\mathbf {i}}
\def \bj {\mathbf {j}}
\def \bk {\mathbf {k}}
\def \bK {\mathbf {K}}
\def \bm {\mathbf {m}}
\def \bn {\mathbf {n}}
\def \bt {\mathbf {t}}
\def \bu {\mathbf {u}}
\def \bv {\mathbf {v}}
\def \by {\mathbf {y}}
\def \bV {\mathbf {V}}
\def \bx {\mathbf {x}}
\def \bw {\mathbf {w}}
\def \b1 {\mathbf {1}}
\def \bz {\mathbf {z}}
\def \bga {\boldsymbol \alpha}
\def \bgb {\boldsymbol \beta}
\def \bgg {\boldsymbol \gamma}
\def \bgs {\boldsymbol \sigma}
\def \bgth {\boldsymbol \theta}

\def \itc {\text{\it c}}
\def \ite {\text{\it e}}
\def \ith {\text{\it h}}
\def \iti {\text{\it i}}
\def \itj {\text{\it j}}
\def \itm {\text{\it m}}
\def \itM {\text{\it M}} 
\def \itn {\text{\it n}}
\def \ithn {\text{\it hn}}
\def \itt {\text{\it t}}

\def \cA {\mathcal{A}}
\def \cB {\mathcal{B}}
\def \cC {\mathcal{C}}
\def \cD {\mathcal{D}}
\def \cE {\mathcal{E}}
\def \cF {\mathcal{F}}
\def \cG {\mathcal{G}}
\def \cH {\mathcal{H}}
\def \mcI {\mathcal{I}}
\def \cK {\mathcal{K}}
\def \cL {\mathcal{L}}
\def \cM {\mathcal{M}}
\def \cN {\mathcal{N}}
\def \cO {\mathcal{O}}
\def \cP {\mathcal{P}}
\def \cQ {\mathcal{Q}}
\def \cR {\mathcal{R}}
\def \cS {\mathcal{S}}
\def \cT {\mathcal{T}}
\def \cU {\mathcal{U}}
\def \cV {\mathcal{V}}
\def \cW {\mathcal{W}}
\def \cX {\mathcal{X}}
\def \cY {\mathcal{Y}}
\def \cZ {\mathcal{Z}}

\def \ga {\alpha}
\def \gb {\beta}
\def \gg {\gamma}
\def \gd {\delta}
\def \ge {\epsilon}
\def \gevar {\varepsilon}
\def \gk {\kappa}
\def \gl {\lambda}
\def \gs {\sigma}
\def \gt {\tau}
\def \gw {\omega}
\def \gz {\zeta}
\def \gG {\Gamma}
\def \gD {\Delta}
\def \gL {\Lambda}
\def \gS {\Sigma}
\def \gW {\Omega}

\def \dim {{\rm dim}\,}
\def \mod {{\rm mod}\;}
\def \rank {{\rm rank}\,}

% These are Brian's
\newcommand{\ds}{\displaystyle}
\newcommand{\vf}{\vspace{\fill}}
\newcommand{\vect}[1]{{\bf{#1}}}
\def\R{\mathbb R}
\def\C{\mathbb C}
\def\CP{\mathbb{C}P}
\def\RP{\mathbb{R}P}
\def\N{\mathbb N}

\def\Sym{\mathrm{Sym}}
\def\Sk{\mathrm{Sk}}
\def\GL{\mathrm{GL}}
\def\Diff{\mathrm{Diff}}
\def\id{\mathrm{id}}
\def\Pf{\mathrm{Pf}}
\def\sll{\mathfrak{sl}}
\def\g{\mathfrak{g}}
\def\h{\mathfrak{h}}
\def\k{\mathfrak{k}}
\def\t{\mathfrak{t}}
\def\OcN{\mathscr{O}_{\C^N}}
\def\Ocn{\mathscr{O}_{\C^n}}
\def\Ocm{\mathscr{O}_{\C^m}}
\def\Ocnz{\mathscr{O}_{\C^n,0}}
%%%\def\E{\mathscr{E}}
%%%\def\dimc{{\dim_{\C}}}
%%%\def\dimC{{\dim_{\C}}}
% TODO: replace with operator
\def\Derlog{\mathrm{Derlog}\,}
\def\expdeg{\mathrm{exp\,deg}\,}

%\begin{document}
\title[Rigidity Properties of the Blum Medial Axis]
{Rigidity Properties of the Blum Medial Axis}
\author[James Damon]{James Damon}

%%% \thanks{(1) Partially supported by the National Science Foundation grant 
%%%  DMS-1105470}
\address{Department of Mathematics, University of North Carolina, Chapel 
Hill, NC 27599-3250, USA
}



\begin{abstract}
We consider the Blum medial axis of a region in $\R^n$ with piecewise smooth boundary and examine its \lq\lq rigidity properties\rq\rq, by which we mean properties preserved under diffeomorphisms of the regions preserving the medial axis.  There are several possible versions of rigidity depending on what features of the Blum medial axis we wish to retain.  We use a form of the cross ratio from projective geometry to show that in the case of four smooth sheets of the medial axis meeting along a branching submanifold, the cross ratio defines a function on the branching sheet which must be preserved under any diffeomorphism of the medial axis with another.  Second, we show in the generic case, along a $Y$-branching submanifold that there are three cross ratios involving the three limiting tangent planes of the three smooth sheets and each of the hyperplanes defined by one of the radial lines and the tangent space to the 
$Y$-branching submanifold at the point, which again must be preserved.  Moreover, the triple of cross ratios then locally uniquely determines the angles between the smooth sheets.  Third, we observe that for a diffeomorphism of the region preserving the Blum medial axis and the infinitesimal directions of the radial lines, the second derivative of the diffeomorphism at points of the medial axis must satisfy a condition relating the radial shape operators and hence the differential geometry of the boundaries at corresponding boundary points.  
\end{abstract} 
\maketitle
\vspace{2ex}
\centerline{\bf Preliminary Version}
\vspace{1ex}
\par
\section*{Introduction}  
\label{S:sec0}
We consider the Blum medial axes of regions $\gW_i \subset \R^n$ with piecewise smooth boundaries $\cB_i$ and examine their rigidity properties (which will be preserved under diffeomorphisms of the regions preserving the medial axes).  The Blum medial axis was introduced in \cite{BN} and its properties for generic regions with smooth boundaries were obtained in \cite{Yo} and \cite{M}, also see \cite{GK} (and more generally for regions with piecewise smooth boundaries in \cite[Part 1]{DG} and 
\cite{DG2}).  Their extensive uses for imaging questions are covered in the book 
\cite{PS}.  \par 
For distinct diffeomorphic regions with homeomorphic medial axes, a basic question is whether there are diffeomorphisms between the regions which preserve various properties of the medial axis structures.  This raises the prospect that there are rigidity properties that must be preserved under a diffeomorphism.  There are several possible versions of rigidity depending on which features of the Blum medial axis we wish to retain.   We list several of these in \S \ref{S:sec1} and ask when there are diffeomorphisms preserving the geometric properties. \par  
For example, the $2D$ regions in Figure \ref{fig.1} appear to be smoothly similar enough that there is a smooth diffeomorphism between them which preserves the medial axes of the regions.  Unfortunately, this is not the case in general due to the presence of a cross ratio of the tangent lines to the four branches at the central point. In this note we examine how rigid mathematical invariants based on the cross ratios associated to various geometric features serve as obstructions to obtaining diffeomorphisms with the additional properties.  In \S \ref{S:sec1} we list several increasingly restrictive conditions on the medial structure which a diffeomorphism might be asked to preserve.  This raises the question in what ways methods such as those developed by Yushkevich et. al. in \cite{Y}, which do obtain diffeomorphisms of regions satisfying 1) in the list, must only provide an approximation to satisfying 2), 3), or 4) in the list?     
\begin{figure}[ht]
\centerline{\includegraphics[width=12cm]{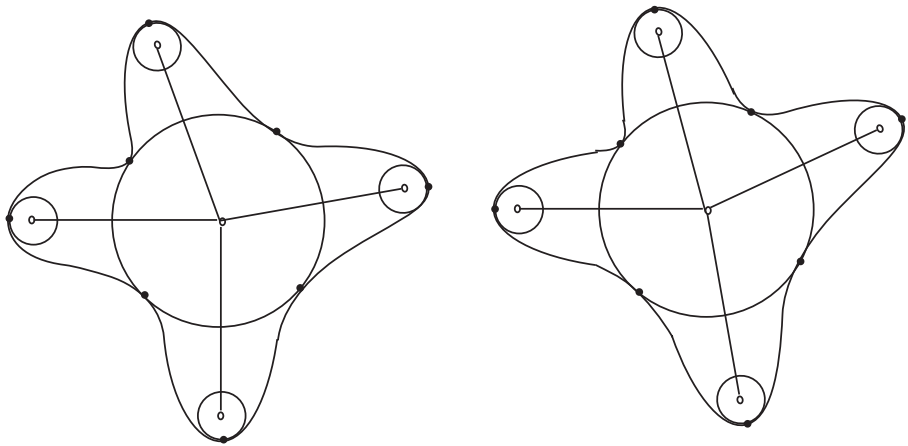}}
\caption{\label{fig.1}  Regions $\gW_i$ with Blum medial axes $M_i$ which are homeomorphic as pairs$(\gW_i, M_i)$.  In addition, the regions $\gW_i$ are diffeomorphic by a diffeomorphism of the boundaries which sends the boundary points corresponding to the four $A_3$ points (where the maxima of curvature of the boundary curves occur)  and the four $A_1^4$ points on each boundary curve (corresponding to the central points) to the others.  A basic question is whether there is a diffeomorphism satisfying 1) and 2) in the list in \S \ref{S:sec1} but which also maps $M_1$ to $M_2$?}  
%% do not have an extension to a medial axis preserving diffeomorphisms of the 
%%% associated regions
\end{figure}  
\par
We consider arbitrary dimensions $n \geq 2$ and use a form of the cross ratio from projective geometry applied to hypersurfaces to show that in the case of four smooth sheets of the medial axis meeting along a branching submanifold, the cross ratio defines a function on the branching submanifold which must be preserved under any diffeomorphism of the medial axis with another.  Second, in the generic case, along a $Y$-branching submanifold there are three cross ratios involving the three limiting tangent spaces of the three smooth sheets together with a hyperplane spanned by one of the radial lines together with the tangent space to the $Y$-branching submanifold at the point.  If the diffeomorphism infinitesimally preserves the radial lines, then we show these three cross ratios must again be preserved.  Moreover, we show that the ordered triple of cross ratios uniquely locally determine the ordered triple of angles between the branching smooth sheets.  \par 
Third, we observe as a result of \cite[\S 3 and \S 5]{D1} that for a diffeomorphism of the region preserving the Blum medial axis and the infinitesimal directions of the radial lines, the second derivative of the diffeomorphism at points of the medial axis must satisfy an algebraic condition relating the radial shape operators, and hence the differential geometry of the boundaries at corresponding boundary points.
\par
As a consequence of these results, if we wish to preserve medial structures under diffeomorphisms of regions, then in general the structures must be allowed to belong to the more general class of \lq\lq skeletal structures\rq\rq (see \cite{D} or \cite{D2}).  Even for these, it will further follow that it may be only possible to have stratawise (for the skeletal sets) diffeomorphisms of the regions.  \par
 Although we develop the results for general $\R^n$, we specifically indicate the form they take for imaging questions for the special cases of regions in $\R^2$ and $\R^3$.
\section{Types of Rigidity Questions for the Blum medial Axis}  
\label{S:sec1}
\par
We consider the Blum medial axis of regions $\gW_i \subset \R^n$ with piecewise smooth boundaries $\cB_i$ and examine their rigidity properties.  There are several possible versions of rigidity depending on what features of the Blum medial axis we wish to retain.  We let $\gW_i$ have Blum medial axis $M_i$ with associated 
multi-valued radial vector field $U_i$ on $M_i$ from points $x \in  M_i$ to the corresponding points $ \psi(x) = x + U_i(x)$ on the boundary $\cB_i$.  Suppose 
that the pairs $(\gW_i, M_i)$, $i = 1, 2$ are at least homeomorphic, and that there 
is a diffeomorphism $\varphi : \gW_1 \simeq \gW_2$.  We may ask several questions 
about whether we can modify $\varphi$ to preserve features of the $M_i$.  \par
\vspace{1ex} 
These might include: \flushpar
{\it Properties Involving Types of Rigidity :} \par
\begin{itemize}
\item[1)] in addition, $\varphi$ maps the points on the boundary $y_j \in \cB_1$ corresponding to the point $x_j \in M_1$ to the points $\varphi(y_j) \in \cB_2$ corresponding to the point $\varphi(x_j) \in M_2$; or 
\item[2)] in addition to 1) that $\varphi$ restricts to a diffeomorphism $\varphi : M_1 \simeq M_2$; or
\item[3)] in addition to 1) and 2), that $d\varphi(x_j)(U_1(x_j)) =  U_2(\varphi(x_j))$ for $x_j \in M_1$ and all values of $U_1(x_j)$; or 
\item[$3^{\prime})$] in addition to 1) and 2), that $d\varphi(x_j)$ at least preserves radial lines, i.e. $d\varphi(x_j)(\langle U_1(x_j)\rangle) =  \langle U_2(\varphi(x_j))\rangle$ for $x_j \in M_1$ and all values of $U_1(x_j)$; or
\item[4)] if $\varphi$ satisfies both 1) and 2), how closely can $\varphi$ satisfy 3) or at least $3^{\prime})$ as well? 
\end{itemize}
While condition 3) would be desirable for a diffeomorphism preserving the full Blum medial structure, we shall see that already satisfying $3^{\prime})$ places significant restrictions on diffeomorphisms.   \par
We consider these properties in both the generic and non-generic cases (where in the later we assume the Blum medial axis still satisfies the conditions for being a skeletal structure as in \cite{D}).  These will also apply to regions with piecewise smooth boundaries as e.g. in \cite[Part I]{DG} or \cite{DG2}.  We will be principally concerned with how the medial structure at branching points restricts the existence of diffeomorphisms preserving the medial structure given by the above conditions.  We do not attempt at this time to determine further specialized conditions for edge points, or in $\R^3$ for fin points, or $6$-junction points.  \par
	For example, in Fig. \ref{fig.1}  are simple regions with nongeneric medial axis structures which apparently should have diffeomorphic deformations between them.  However, in fact, the two competing conditions of preserving the medial axis and being a diffeomorphism on the entire interior region are completely incompatible.  We will see that no such diffeomorphism is possible.  In Fig. \ref{fig2.1} we illustrate a generic medial axis structure at a branch point in $\R^2$.  The inclusion of the radial vectors at the branch point provides sufficient additional data so that again we identify obstructions to the existence of diffeomorphisms satisfying condition $3^{\prime})$.  We also explain how the results extend to higher dimensions.  
\par
	For diffeomorphisms of regions preserving the medial axes structures as in 3) we also explain a further obstruction involving a second order condition on the diffeomorphism at points of the medial axis which relates a specific second derivative to the Òradial shape operatorsÓ for each of the medial axes.  It gives a specific algebraic relation involving the radial shape operators and a Òradial distortion operatorÓ defined from the second derivative of the diffeomorphism.  This was derived in \cite[Thm 5.4]{D1}, and it is also briefly discussed in the author\lq s chapter \cite[\S 3.3.3]{PS}).  \par

\section{Infinitesimal Properties of the Blum Medial Axis at Branch Points}  
\label{S:sec2}
\par
Before considering differentiable invariants of the Blum medial axis at branch points, we first obtain several simple relations between the angles of the tangent spaces and the radial vectors.  
\par 
We consider a point $x$ on an $A_1^3$ stratum of a generic Blum medial axis.  This is a branch point in the 2D case and a point on a $Y$-branch curve for the 3D case.  We first consider the case of a region $ \gW \subset \R^2$. The corresponding results for general $\R^n$ follows by using angles between the limiting hyperplanes tangent to the three smooth sheets at the branch point.  For the three branches there are unique limiting tangent lines $L_i$, $ i = 1, \dots, 3$, forming successive angles $\theta_i$ between the successive lines, given on the counterclockwise direction, as illustrated in Fig. \ref{fig2.1}.  We also consider the radial vectors $U_i$ from $x$ to points on the boundary and in the region corresponding to the angle $\theta_i$.  Third, we let $\ga_j$ denote the angle from the line $L_i$ to the next radial vector $U_i$ going counterclockwise toward the line $L_j$.  Then, there is the following relation between the angles.
\begin{Lemma}
\label{Lem2.1}
At a generic $Y$ branching point, the angles as above satisfy the relations 
$\ga_i = \pi - \theta_i$, $i = 1, \dots , 3$.
\end{Lemma}
\begin{proof}
We use the property of the Blum medial axis that the angles from the tangent line 
$L_i$ to the radial vectors on each side of $L_i$ are equal, to obtain the equations
\begin{align}
\label{Eqn2.1}
\theta_1 \,\, &= \,\, \ga_2 + \ga_3  \notag  \\
\theta_2 \,\, &= \,\, \ga_3 + \ga_1  \notag  \\
\theta_3 \,\, &= \,\, \ga_1 + \ga_2 
\end{align}
Using $\theta_1 + \theta_2 + \theta_3 = 2\pi$, we easily verify that these have unique solutions $\ga_i = \pi - \theta_i$, $i = 1, \dots , 3$.
\end{proof}
\par  
%%% \vspace{20ex}
\begin{figure}[ht]
\centerline{\includegraphics[width=6cm]{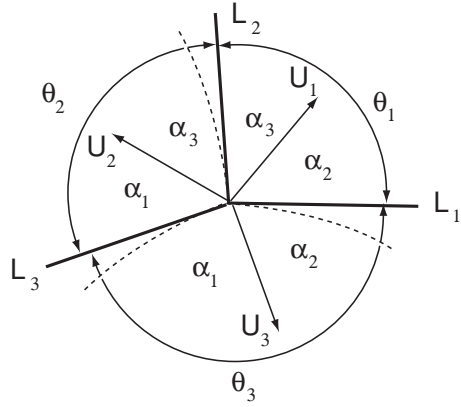}}
\caption{\label{fig2.1}  The configuration of tangent lines, radial vectors and angles at a generic branch point of the Blum medial axis of a region in $\R^2$.}  
\end{figure}  
%%% Figure 2.1:  The configuration of tangent lines and radial vectors at a generic 
%%% branch point of the Blum medial axis of a region in $\R^2$. \par
%%% \vspace{1ex}
\par
Second, suppose that there are four branch curves with tangent lines $L_i$, $i = 1, \dots , 4$ in counterclockwise order as in Fig. \ref{fig2.2}, with the radial vectors $U_i$ ordered as above, and the angles $\gb_i$ are from $L_i$ to the next radial vector $U_i$ in the counterclockwise direction.  
\begin{Lemma}
\label{Lem2.2}
At a generic $Y$ branching point with four branch curves, the angles as above satisfy the following relations:  
\begin{equation}
\label{Eqn2.2}
  \theta_1 + \theta_3\,\, = \,\,  \theta_2 + \theta_4 \, ; 
\end{equation}
and the values for $\gb_i$ are given in parametrized form by 
\begin{equation}
\label{Eqn2.3}
 (\gb_1, \gb_2, \gb_3, \gb_4) \,\, = \,\, (\theta_4, \theta_1 - \theta_4, \theta_3, 0) + t(-1, 1, -1, 1) 
\end{equation}
\end{Lemma}
\begin{proof}
Then, we use the property of the Blum medial axis that the angles from the tangent line 
$L_i$ to the radial vectors on each side of $L_i$ are equal, to obtain the equations
\begin{align}
\label{Eqn2.4}
\theta_1 \,\, &= \,\, \gb_1 + \gb_2  \notag  \\
\theta_2 \,\, &= \,\, \gb_2 + \gb_3  \notag  \\
\theta_3 \,\, &= \,\, \gb_3 + \gb_4  \notag  \\
\theta_4 \,\, &= \,\, \gb_4 + \gb_1 
\end{align}
Then, using a standard method such as Gaussian elimination, we see that 
\eqref{Eqn2.2} is a necessary condition for a solution and then we may solve these for the $\gb_i$ to obtain the above solutions given by \eqref{Eqn2.3}.
\end{proof}
\begin{Remark}
\label{Rem2.2a}
We observe that a consequence of Lemma \ref{Lem2.2}, given fixed angles $\theta_i$, $i = 1, \dots , 4$, there is a family of angles for the radial vectors consistent with the Blum condition.  As a consequence, even if the diffeomorphism preserves the Blum medial axis, there is a continuous family of consistent angles for radial vectors, so it would not in general preserve the directions of the radial vectors.    
\end{Remark}

\begin{figure}[ht]
\centerline{\includegraphics[width=6cm]{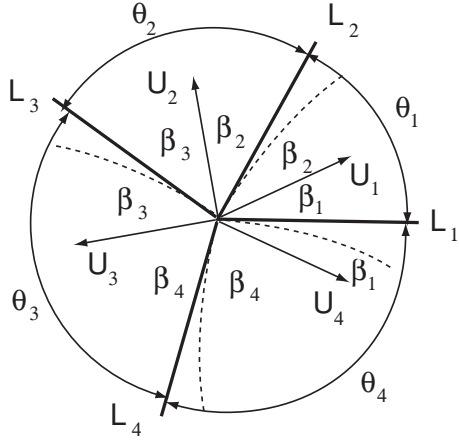}}
\caption{\label{fig2.2}  The configuration of tangent lines, radial vectors and angles at a branch point of the Blum medial axis where four curves meet for a region in $\R^2$.}  
\end{figure} 
%%% Figure 2.2: The configuration of tangent lines and radial vectors at a branch 
%%% point of the Blum medial axis where four curves meet for a region in $\R^2$. 
\par
These arguments extend to higher dimensions by the properties of the Blum medial axis in the generic case or in the nongeneric case for a skeletal structure provided at any smooth point $x \in M$ the pair the radial vectors $U_1(x)$ and $U_2(x)$ at $x$ make equal angles with $T_x M$ and satisfy $U_1(x) - U_2(x)$ is normal to $T_x M$.  This continues to hold in the limit for a singular point $x$ where we approach $x$ along a smooth sheet of $M$.  Let $\gG$ denote a codimension $2$ branching stratum with $x \in \gG$, and for a smooth stratum $M_i$ with $\gG$ in its closure, we let $T_xM_i$ denote the limiting tangent plane to $M_i$ at $x$.  By the properties of Blum medial axes in the generic case or skeletal structures, $T_x \gG \subset T_xM_i$.  \par 
Then, let $P$ denote the orthogonal plane to $T_x \gG$ at $x$.  For $P$ we choose an orthonormal basis $\{ e_1, e_2\}$ with $e_1 \in T_x M_i$ and so $e_2 \perp T_x M_i$.  
Since $U_1(x) - U_2(x)$ is normal to $T_x M_i$ we may write them $U_1 = v + a e_1 + b e_2$ and $U_1 = v + a e_1 - b e_2$ with $v \in T_x \gG$.  The hyperplanes $\Pi _i$ spanned by $T_x \gG$ with each $U_i$ therefore also contain, respectively, $a e_1 \pm b e_2$.  As $P$ is also orthogonal (and hence transverse) to $T_x M_i$, $M_i \cap P$ is a curve $\gg$ with limiting point $x$ and limiting tangent line $T_x \gg = P \cap T_x M_i$ which is spanned by $e_1$.  It follows that the vectors $a e_1 + b e_2$ and $a e_1 - b e_2$ in $P$ make equal angles with $e_1$ and hence the tangent line 
$T_ \gg$.  However, the angles between the hyperplanes $\Pi_i$ and $T_x M_i$ are given by these angles in the orthogonal plane $P$. \par  
Since we may repeat this argument for each smooth sheet whose closure contains the stratum $\gG$, we arrive in the generic case with the configuration of line and curves in the orthogonal plane $P$ as in Fig. \ref{fig2.1}, and in the nongeneric case with four smooth sheets meeting along a branching stratum $\gG$ the configuration in $P$ as in Fig. \ref{fig2.2}. 
\par
\section{Cross Ratio}  
\label{S:sec3}
Next, we recall the properties of the cross ratio, which is an invariant of four ordered points in a projective line, and indicate how it applies to four hypersurfaces $H_i$ in 
$\R^n$ which contain a common codimension $2$ subspace $L$. \par
\subsection*{Cross Ratio for Points in a Projective Line } \hfill  
\par
First the cross ratio is generally defined for four distinct points $\{z_1, \dots , z_4\}$ in a complex projective line $\CP^1$. The cross ratio is defined by 
\begin{equation}
\label{Eq2.1} 
R(z_1, \dots , z_4) \,\,  = \,\,  \frac{(z_1 - z_4)(z_3 - z_2)}{(z_1 - z_2)(z_3 - z_4)}  \, .
\end{equation}
As $\CP^1$ can be viewed as the complex plane with point added at $\infty$, this value is defined if no $z_i = \infty$, and there is an assignment in the case one is $\infty$ by taking a limit as the $z_i \to \infty$.  \par
We note that this depends on the order of the $z_i$.  If the order is changed and we let $\gl$ denote the cross ratio in \eqref{Eq2.1}, then after permuting the order of the points, we obtain five additional values obtained from $\gl$, under the operations: 
\begin{equation}
\label{Eq2.1a} 
\qquad \frac{1}{\gl}, \qquad 1 - \gl, \qquad \frac{1}{1 - \gl}, \qquad \frac{\gl - 1}{\gl}, \qquad \frac{\gl} {\gl - 1}  
\end{equation}  
These are the set of values obtained under the action of the finite group generated by the two transformations $\gl \mapsto \frac{1}{\gl}$ and $\gl \mapsto 1 - \gl$.  This group is isomorphic to the permutation group on three letters $\cS_3$.
Furthermore, it is a basic fact from projective geometry that for any collection of four ordered distinct points, the cross ratio is invariant under a projective transformations of $\CP^1$.  Now points in $\CP^1$ can be identified with lines in $\C^2$ through the origin.  Then the corresponding basic fact from projective geometry states that for 
any collection of four ordered distinct lines in $\C^2$ through the origin, the cross ratio 
is invariant under any invertible linear transformation of $\C^2$.  \par
In the case that the points are real, then the cross ratios are real and there is a corresponding statement for lines in $\R^2$.  These and other properties may be found, for example, in \cite[\S 3.3]{Az}.  \par
To compute the cross ratio of four lines in $\R^2$, suppose they are rotated so none lies along the $y$-axis.  Then they all have the form $y = a_i x$, $ i = 1, \dots , 4$.  If we let the $y$-axis be the line at infinity, and the line $x = 1$ corresponds to the complementary affine line, then the line $y = a_i x$ corresponds to the point $y = a_i$, and the cross ratio is given by $\frac{(a_1 - a_4)(a_3 - a_2)}{(a_1 - a_2)(a_3 - a_4)}$.
\par
\subsection*{Generalized Cross Ratio for Hyperplanes in $\C^n$} \hfill  \par
The notion of cross ratio extends to hyperplanes in $\C^n$.  Let $H_i \subset \C^n$, 
$i = 1, \cdots , 4$, denote hyperplanes containing the codimension $2$ subspace $L$.  If $L^{\perp}$ denotes the orthogonal complement to $L$, then each 
$H_i \cap L^{\perp} = L_i \subset L^{\perp} \simeq \C^2$.  Thus, the four ordered lines have a cross ratio $R(L_1, \dots , L_4)$; and a permutation of the hyperplanes gives a set of six values as above.  Moreover, if instead of $L^{\perp}$, we chose a plane $\Pi$ through the origin and transverse to $L$, then we obtain a second set of ordered lines $L_i^{\prime} = H_i \cap \Pi = L_i \subset \Pi \simeq \C^2$.  If we consider the restriction to $\Pi$ of the orthogonal projection of $\C^n$ to $L^{\perp}$ along $L$, then it gives an isomorphism $\Pi \simeq L$ which sends $L_i^{\prime} \mapsto L_i$.  Hence by the invariance of the cross ratio under invertible linear transformations, we obtain the same value for the cross ratio using either $L^{\perp}$ or $\Pi$.  Hence, the cross ratio is an intrinsic invariant of the four ordered hyperplanes and is invariant under invertible linear transformations of $\C^n$.  Again, there is a corresponding result for hyperplanes in 
$\R^n$.  In fact, what we really are saying is that the set of hyperplanes containing $L$ forms a projective line in the dual projective space $\CP^{n-1\, *}$ and the cross ratio is the invariant for that projective line.\par
	In the next sections we see the consequences for rigidity properties of the cross ration and its generalization for hyperplanes. 

\section{Rigidity Properties for Four Smooth Strata Meeting Along a Branching Submanifold}  
\label{S:sec4}
\par
%% We consider two regions $\gW_i$, $i = 1, 2$, as in Fig. \ref{fig.1} with smooth 
%%% boundaries $\cB_i$, and medial axes $M_i$.  if there is a diffeomorphism 
%%% $\varphi : \gW_1 \simeq \gW_2$ 
\par
We begin by considering regions in $\R^n$.  Suppose that the region $\gW \subset 
\R^n$ has a nongeneric Blum medial axis $M$ which together with its multivalued radial vector field still satisfies the condition for being a skeletal structure.  In particular, we suppose there is a codimension $2$ branching stratum $\gG \subset M$ along which four smooth (codimension one) medial sheets $S_j$, $j = 1, \dots, 4$ meet, and moreover for any $x \in \gG$ there are unique limiting tangent planes $T_x S_j$ with $T_x \gG \subset T_x S_j$ for each $j$.  We then define a cross ratio invariant for this situation.  For the point $x \in \gG$ we have four hyperplanes $T_i = T_xS_i$, and they each contain the codimension two subspace $T_x \gG$.  Hence, by the arguments in \S \ref{S:sec3}, the four distinct hyperplanes containing the common subspace $T_x \gG$ have a real-valued cross ratio.  Allowing different ordering again gives the six possible real values as in \eqref{Eq2.1a}.  We can thus give a 
well-define {\em cross ratio map}  $\chi : \gG \to \RP^1/\cS_3 \simeq S^1$, where the target space consists of the sets of corresponding cross ratio values. 
\par  
Suppose that the two regions $\gW_i \subset \R^n$ both have medial axes which each contain a submanifold $\gG_i$ as above, along which four smooth medial sheets $S^{(i)}_j$, $j = 1, \dots, 4$ meet.  Then there is the following strong rigidity condition on a diffeomorphism.  
\begin{Thm}[Strong Generalized Rigidity]
\label{Thm4.2}
Suppose there is a diffeomorphism $\varphi$ defined in a neighborhood of $\gG_1$ which sends $\gG_1$ to $\gG_2$ and the medial sheets $S^{(1)}_j$ of $\gW_1$ to those $S^{(2)}_j$ of $\gW_2$.  Then, for $\chi_i$, $i = 1, 2$ denoting the cross ratio maps for $\gG_i$, we must have $\chi_2\circ \varphi = \chi_1$.  
\end{Thm}
\par
\begin{proof}
For any point $x \in \gG_1$, the derivative $d_x\varphi$ will send the limiting tangent planes $T^{(1)}_j = T_x S^{(1)}_j$ to $T^{(2)}_j = T_{\varphi(x)} S^{(2)}_j$.  Also, the four tangent planes $T^{(1)}_j$ contain the common codimension $2$ tangent space $T_x \gG_1$, and $T^{(2)}_j$, the common codimension $2$ tangent space 
$T_ {\varphi(x)} \gG_2$.  \par
Now by the above arguments, through a point $x$, resp. $\varphi(x)$, there is again a set of cross ratios for each ordering.  Thus, the sets of four limiting tangent planes gives rise to a set of six values as in \eqref{Eq2.1a}.  Since $d\varphi(x) (T^{(1)}_j) = T^{(2)}_j$ for $j = 1, \cdots, 4$, it follows by the invariance of the cross ratios that the two sets of cross ratios must agree.  Thus, the induced maps $\chi_2\circ \varphi$ and 
$\chi_1$ must have the same values for each point $x \in \gG_1$.  
\end{proof}
\par
The cross ratio is a \lq\lq rigid\rq\rq\, invariant for four lines in 
$\R^2$ meeting at a point. As the four tangent hyperplanes vary continuously, the cross ratio maps $\chi_i$ on the branching stratum are varying and so the rigidity has a very strong form that they must be matched exactly by the diffeomorphism for each pair of points.  \par
The simplest form of this is for regions in $\R^2$. 
\vspace{1ex}
\begin{Example}
\label{Ex4.1}
We consider two regions $\gW_i \subset \R^2$, $i = 1, 2$, as in Fig. \ref{fig.1} with smooth boundaries $\cB_i$, and medial axes $M_i$.  There is a diffeomorphism between the boundaries which preserves the corresponding medial data on the boundary.  By this medial data we mean the four points of types $A_3$ and four points of types $A_1^4$ on the boundary corresponding to the points on the medial axis with their corresponding medial type.  This diffeomorphism satisfies properties 1) and 2).  We consider the tangent lines to the four curve branches meeting at the center ($A_1^4$) point.  Each set of these four lines have, up to a choice of ordering, a {\em set of cross ratios} $\{\gl_j^{(i)}$\}, $i = 1, 2$.  
\par  
Suppose we have two medial axes each consisting of four branch curves $\{\gg^{(i)}_j : j = 1, \dots , 4\}$ for $i = 1, 2$.  Let $\ell^{(i)}_j$ denote the tangent line to $\gg^{(i)}_j$ at the corresponding center point.   There is the following special case of Theorem \ref{Thm4.2}.
\end{Example}
\begin{Corollary}
\label{Cor4.1}
If the sets of tangent lines $\ell^{(i)}_j$, $j = 1, \dots , 4$, give distinct sets of six values for each medial axis, then there does not exist a diffeomorphism 
$\varphi : U_1 \to U_2$ defined between the neighborhoods $U_i$ of the center points which maps one set of the four branch curves $\gg^{(1)}_j$ to the other set of the four branch curves $\gg^{(2)}_j$ (after renumbering).
\end{Corollary}
\begin{proof}
 In this case there is only a single set of cross ratios.  If they disagree at the branch points then there cannot be a diffeomorphism preserving the medial axes.  
\end{proof}
\par
In the preceding situation, suppose we wish to define a diffeomorphism preserving the medial axis.  Suppose the diffeomorphism is constructed to send three of the four curves to three of the curves for the second configuration, but the cross ratios are significantly different.  Then the image of the fourth curve will differ significantly from the fourth curve.
\begin{Remark}
\label{Rem4.2}
Since the cross ratio value may be any real number it follows that that given any two random choices of sets of four distinct lines, with probability $1$, the sets of cross ratios will be distinct.  Thus, for arbitrary random choices of regions as in Figure \ref{fig.1}. there will be no diffeomorphism between the regions preserving the medial axes. 
\end{Remark}
\par 
\begin{figure}[ht]
\centerline{\includegraphics[width=8cm]{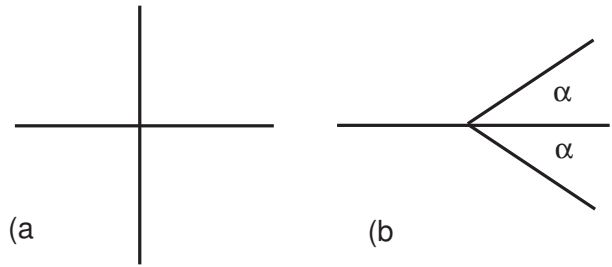}}
\caption{\label{fig4.3}  Illustrating for two Blum medial axes for regions in $\R^2$ where four curves meet at a branch point, the nonexistence of a local diffeomorphism mapping one Blum medial axis to the other.}  
\end{figure} 
\begin{Example}
\label{Ex4.3}
	We consider the configuration of half lines in Fig. \ref{fig4.3} which represent medial Blum medial axes of regions $\gW_1$ and $\gW_2$ in a neighborhood of a branch point.  Although the configurations are degenerate so the cross ratio does not apply.  However, we illustrate how far a local diffeomorphism must distort the radial structure.  Consider the local diffeomorphism  $\varphi : \gW_1 \simeq \gW_2$ which maps a neighborhood of the branch point (denoted $0$) of one to that of the other.  	Suppose $\varphi$ maps the $x$-axis to the $x$-axis and the positive $y$-axis to the line in the first quadrant making the angle $\ga$ with the positive $x$-axis.  Then, the derivative $d\varphi(0)$ is a linear transformation sending the $x$-axis to the $x$-axis, so $d\varphi(0)$ sends $e_1 = (1, 0) \mapsto (s, 0)$ for some $s > 0$. Also, as 
$\varphi$ sends the positive $y$-axis to the angled line in the first quadrant, it must send $e_2 = (0, 1) \mapsto (r \cos(\ga) , r \sin(\ga))$ for some $r > 0$.  This completely determines the derivative $d\varphi(0)$.  \par
	Then, the negative $y$-axis is being sent to a curve, which by the linearity of 
$d\varphi(0)$,  has tangent vector at the origin given by $d\varphi(0)((0, -1)) =  (-r \cos(\ga) , -r \sin(\ga))$.  Thus, the curve initially heads into the third quadrant making an angle $\ga$ with the negative $x$-axis.  For example, if $\ga = \pi/3$, then the curve will initially make an angle of $2\pi/3$ with the second angled line in the fourth quadrant that we would like it to map to.  Hence, it must make a large turn to even go in the roughly correct direction.  
\end{Example}
\par  
This example was investigated by Yushkevich who using the algorithm in \cite{Y} to obtain a video which shows that the $y$-axis of the source actually maps to a parabolic curve tangent to the $y$-axis.  Thus, initially two of the medial curves move in a significantly different direction from those of the target medial curves. The preceding results show that any attempt to construct a local diffeomorphism must encounter a similar phenomenon when there are different cross ratios.  We show in the next section that if we include the radial vectors in the Blum structure, that a similar phenomenon occurs in the generic case. 
\par
\begin{Remark}
\label{Rem4.3}
Because there are many different ways in which such a configuration may occur within a nongeneric medial axis, it follows that there are many circumstances where there is no diffeomorphism preserving the medial axis.  In particular, if there are more than four smooth sheets meeting along a branching submanifold, then successively choosing four such sheets gives a set of cross ratio invariants.  While such a condition is not generic, we next consider further invariants which incorporate the radial vectors at the branch points.
\end{Remark} 
\par 

\section{Rigidity of Infinitesimal Properties for Smooth Strata Meeting Generically Along a Branching Submanifold}  
\label{S:sec5}
\par
We next consider the effect of diffeomorphisms in the generic case.  Generically for regions in $\R^2$, $3$ medial curves meet at a branch point; or for regions in $\R^3$, three medial surfaces meeting along a $Y$-branch curve, and quite generally for a generic region in $\R^n$ there can be three smooth sheets of the medial axis meeting along a branching codimension $2$ subspace $\gG$.  For a branch point $x \in \gG$, there are three radial vectors from $x$ to the boundary.  Each of the radial vectors from $x$ determine a radial line in the complementary region corresponding to the point on the boundary.  This line along with tangent space $T_x \gG$ determines a hypersurface
 in $\R^n$. This hyperplane together with the other three limiting tangent hyperplanes of the smooth sheets again give four hyperplanes containing the codimension $2$ subspace $T_x \gG$.  These have a cross ratio.  We can compute it by intersecting the hyperplanes with a plane $P$ through $x$ and transverse to $T_x \gG$.  We use the notation from \S \ref{S:sec2}, we compare the cross ratios formed from the three lines $L_i$, $i = 1, \dots ,3$ obtained by intersecting with $P$ the tangent hyperplanes and the line $\widetilde{L}$, determined by intersecting with $P$ the hyperplane containing $T_x \gG$ and one of the radial vectors ($U_j$ will be understood).  Thus, there are three cases. \par
Referring to Fig. \ref{fig2.1}, we first determine the four values $a_i$ for the ordered lines $L_1, L_2, L_3, \widetilde{L}$ given in the form $y = a_i x$ where we rotate so that $L_1$ is the $x$-axis, we obtain \eqref{Eqn5.1}.
\begin{equation}
\label{Eqn5.1}
(a_1, a_2, a_3, a_4) \,\, = \,\, (0, -\tan(\theta_1), \tan(\theta_3), \tan(\theta_2))  
\end{equation}
Likewise, for the other two cases we rotate so that $L_2$, resp. $L_3$ is the $x$-axis, for the corresponding lines $\widetilde{L}$ for $U_2$, resp. $U_3$.  Then, we obtain the values for the $a_i$ given by \eqref{Eqn5.2}. 
\begin{equation}
\label{Eqn5.2}
 (0, -\tan(\theta_2), \tan(\theta_1), \tan(\theta_3)) \quad \text{resp.} \quad (0, -\tan(\theta_3), \tan(\theta_2), \tan(\theta_1)) \, .
\end{equation}
We observe the the three sets of values are successively transformed by the transformation 
$$ (0, -c_1, c_2, c_3) \,\, \mapsto \,\, (0, -c_3, c_1, c_2) \, .$$
\par  
However, the set of six cross ratios for a set of values $(0, -c_1, c_2, c_3)$ is not invariant under this transformation.  Thus, in general the sets of cross ratios will be distinct for the three radial vectors.  We illustrate this with an example.
\begin{Example}
\label{Ex5.3}
We consider the case of the angles $(\theta_1, \theta_2, \theta_3) = (\frac{2}{3}\pi, \frac{5}{9}\pi, \frac{7}{9}\pi)$.  We obtain the set of values in \eqref{Eqn5.1} and \eqref{Eqn5.2} with corresponding cross ratios $\gl_i$ to be
\begin{align}
\label{Eqn5.4}
[0, 5.671281833, -1.732050808, -0.8390996312] \quad \gl_1 \,\, &= -1.226681596  \notag  \\
[0, 0.8390996312, -5.671281833, -1.732050808] \quad \gl_2 \,\, &= -3.411474126  \notag  \\
[0, 1.732050808, -0.8390996312, -5.671281833] \quad \gl_3 \,\, &= 1.742227197
\end{align}
\par
Then, the corresponding sets of six cross ratio values given by the cross ratio and the other five values obtained by permuting the order of the values as in \eqref{Eq2.1a} are given by
\begin{align}
\label{Eqn5.5}
&[-1.226681596, -0.8152074697, 2.226681596, 0.4490987853, 1.815207470 0.5509012147]  \notag  \\
& [-3.411474126, -0.2931284140, 4.411474126, 0.2266815970, 1.293128414, 0.7733184030] \notag  \\
&[1.742227197, 0.5739779529, -0.742227197, -1.347296359, 0.4260220471, 2.347296359] \notag  \\
&\, 
\end{align}
We see these are distinct sets of cross ratio values.  
\end{Example}
\par
Consequently, we have the following rigidity theorem in the generic case for the Blum structure at a branch point.  We consider regions $\gW_i$, $i = 1, 2$, with generic 
Blum medial axes.  We suppose that there is a local diffeomorphism $\varphi$ from a neighborhood of a branch point $x \in \gG_1$ for $\gG_1$ the $A_1^3$ stratum for 
$\gW_1$, to a neighborhood $\varphi(x) \in \gG_2$ for $\gG_2$ the $A_1^3$ stratum 
for $\gW_2$.  We now let $\chi^{(1)}_j : \gG_1  \to S^1$ denote the corresponding cross ratio for the three limiting tangent spaces at $x$ to the smooth sheets and the hyperplane defined by the radial vector $U_j$, with $\chi^{(2)}_j : \gG_2  \to S^1$.  
the corresponding cross ratios for $\gG_2$.
\begin{Thm}[Rigidity for Generalized Blum Structures]
\label{Thm5.6}
Suppose the diffeomorphism $\varphi$ defined in a neighborhood of the branch point $x \in \gG_1$ which sends $\gG_1$ to $\gG_2$, sending the medial sheets $S^{(1)}_j$ of $\gW_1$ to those $S^{(2)}_j$ of $\gW_2$, and also $d\varphi(x)$ preserves radial lines, i.e.  $d\varphi(x)(\langle U_i(x)\rangle) = \langle U_i^{\prime}(\varphi(x))\rangle$ for the radial vectors $U_i$ on $\gG_1$ and $U_i^{\prime}$ on $\gG_2$.  Then, for $\chi^{(i)}_j$, $j = 1, 2, 3$ denoting the cross ratio maps for $\gG_i$, we must have $\chi^{(2)}_j\circ \varphi = \chi^{(1)}_j$ for $j = 1, 2, 3$. 
\end{Thm}
\par 
The proof follows by an analogous argument to that for Theorem \ref{Thm4.2} using instead the cross ratios for the radial vectors.  It has as a corollary the behavior of diffeomorphisms between regions with generic Blum medial axes. 
\begin{Corollary}
\label{Cor5.8}
Let $\varphi : \gW_1 \to \gW_2$ be a diffeomorphism between generic regions of $\R^n$, which maps the Blum medial axis of $\gW_1$ to that of $\gW_2$.  
Suppose for corresponding branch points $x \in \gW_1$ and $x^{\prime} = \varphi(x) \in \gW_1$ these sets of cross ratios for the Blum structures of the two regions are different.  Then the diffeomorphism will map the radial lines in $\gW_1$ from $x$ to curves from $x^{\prime}$ in $\gW_2$ whose tangent lines at $x^{\prime}$ differ from the radial lines at $x^{\prime}$.  Thus, the diffeomorphism $\varphi$ will distort the radial structure at such branch points.  
\end{Corollary}
\par
In particular, this gives criteria for regions in $\R^2$ at branch points, and for regions in $\R^3$ at points along $Y$-branch curves.  Thus, for regions in either $\R^2$ or $\R^3$, diffeomorphisms between regions which preserve the medial axis will either be severely restricted in the nongeneric case by the set of cross ratios for the medial sheets at branch points, or in the generic case it can map the medial axis, but if the set of cross ratios differ, it will deform the radial structure.  Thus, since the cross ratios control whether such diffeomorphisms can match non-generic templates to similar non-generic target shapes accurately, the question is whether there is a finite bound on how closely they can match a target shape when the set of cross ratios do not agree. \par

\subsection*{Local Uniqueness of Angles from Triples of Cross Ratios} \hfill 
\par
We conclude this section by explaining how almost all triples of allowable angles at a generic branching point, the three cross ratios locally uniquely determine the triple of angles.  The set of allowable angles $\bgth = (\theta_1, \theta_2, \theta_3) \in (0, \pi)^3$ satisfies $\theta_1 + \theta_2 + \theta_3 = 2\pi$.  By \lq\lq almost all\rq\rq \, we mean that it is true on a non-empty open subset of full $2$-dimensional measure in this subspace.  Then, the local uniqueness has the following form. 
\begin{Thm}
\label{5.9}
There is an open set having full $2$-dimensional measure in the subspace of 
$(0, \pi)^3$ consisting of allowable triples $\bgth$, such that the corresponding triple 
of cross ratios uniquely determines $\bgth$ among neighboring triples 
$\bgth^{\prime}$ in a neighborhood of $\bgth$.
\end{Thm}
\par
\begin{proof}
There are two steps.  First we define a Zariski open subset of $(0, \pi)^3$ (which is the complement of a set of algebraic subsets) on which is defined the triple cross ratio map.  To define the Zariski open subset, we first remove the subset where some $\theta_i = \frac{\pi}{2}$ or some $\theta_i = \theta_j$ for $i \neq j$.  The resulting Zariski open subset we denote by $\mcI_3$.  
We will further restrict to the subset $\cQ_3 \subset \mcI_3$ satisfying $\theta_1 + \theta_2 + \theta_3 = 2\pi$.  This Zariski open subset $\cQ_3 $ has full $2$-dimensional measure in $\cQ_3$.  Second, we consider on this open subset the composition of the map with coefficient functions $\tan(\theta_i)$ and the cross ratio map using the three cross ratios of the $4$-tuples given in \eqref{Eqn5.1} and \eqref{Eqn5.2}.  We show it has rank $2$ off a closed analytic subset, whose complement $\cU$  still has full 
$2$-dimensional measure. Thus, for any point $\bgth \in \cU$, the composition is an immersion.  Hence, there is a neighborhood $\cU^{\prime}$ of $\bgth$ on which the composition is an embedding.  It follows that the cross ratio values uniquely determine the triple angle $\bgth^{\prime}$ for all $\bgth^{\prime} \in \cU^{\prime}$.
 \par 
It remains to show the stated properties of the mapping.  We first define a series of maps to give the triple cross ratio map on $\mcI_3$.  
\begin{equation}
\label{Eqn5.6}
\begin{CD}
{\cQ_3}  @>{\iti}>> {\mcI_3} @>{\tau}>> {(\R \backslash \{0\})^3 \backslash (\gD \R)^{(3)}} 
@>{\itc}>> {(\R \backslash \{0\})^3}  @>{\cL}>> {(\R \backslash \{0\})^3}  
\end{CD}
\end{equation}
Here, for any space $X$, the {\em generalized $k$-diagonal} is defined by $(\gD X)^{(k)} = \{ (x_1, \dots , x_k) \in X^k : x_i = x_j \text{for some } i \neq j\}$.  
In \eqref{Eqn5.6}, $\iti$ denotes inclusion.  Next, the mapping $\tau(\theta_1, \theta_2, \theta_3) = (\tan(\theta_1), \tan(\theta_2), \tan(\theta_3))$.  Then, $\tau$ is an analytic diffeomorphism as $\tan$ defines an analytic diffeomorphism $(0, \pi) \backslash \{\frac{\pi}{2}\} \simeq  \R \backslash \{0\}$.  Third, for $(b_1, b_2, b_3) \in (\R \backslash \{0\})^3 \backslash (\gD \R)^{(3)}$, $\itc$ is defined by the triple of cross ratios
$$ \itc(b_1, b_2, b_3) \,\, =  \,\, ((R(0, -b_1, b_3, b_2), R(0, -b_2, b_1, b_3), R(0, -b_3, b_2, b_1)) . $$
Lastly, $\cL$ is defined by $\cL(y_1, y_2, y_3) = (\ln(|y_1|), \ln(|y_2|), \ln(|y_3|))$. 
 \par
Then, we observe that the composition $\itc \circ \tau \circ \iti$ is the triple cross ratio map and is an analytic map on $\cQ_3$.  Hence, the set of points where it has rank 
$< 2$ is an analytic Zariski closed subset, which if not all of $\cQ_3$, has measure zero.   It follows that the cross ratio values uniquely determine the triple angle 
$\bgth^{\prime}$ for all $\bgth^{\prime} \in \cU^{\prime}$.
	Lastly, as $\iti$ is an embedding, $\tau$ is a diffeomorphism, and $\cL$ is everywhere a local diffeomorphism, it is sufficient to show that $\cL \circ \itc$ restricted to $d(\tau \circ \iti)(\bgth)(T_{\bgth}\cQ)$ has rank $2$ for some $\bgth \in \cU$.  Since each cross ratio is homogeneous of degree $0$, the Euler vector field $(b_1, b_2, b_3)$  at each point $(b_1, b_2, b_3)$ is in the kernel of $d\itc$.  Also, the composition $\cL\circ \itc$ has derivative with entries rational functions and is easily seen to have rank $2$ on a Zariski open set.  Thus, the  Euler vector field actually spans the kernel of $d\itc$ on a Zariski open set.  Thus, the rank of the composition will be $2$ at $\bgth = (\theta_1, \theta_2, \theta_3)$ in the Zariski open set unless the Euler vector field belongs to $d(\tau \circ \iti)(\bgth)(T_{\bgth}\cQ)$.  Third, the image of the tangent space is seen to be spanned by 
$(1 + \tan^2(\theta_1), -(1 + \tan^2(\theta_2), 0)$ 
and $(1 + \tan^2(\theta_1), 0, -(1 + \tan^2(\theta_3))$.  Then, these two vectors together with the Euler vector field will form a determinant which is non-zero on the complement of an algebraic subset giving algebraic conditions on the $\tan(\theta_i)$.  Thus, on an analytic Zariski closed subset of $\cQ$, the composition has rank $2$.  \par
This completes the proof.
\end{proof}
\begin{Remark}[Conjecture/Problem]
\label{Rem5.10}
In fact, there may well be a stronger global form of Theorem \ref{5.9} that the triple cross ratio, as an ordered triple uniquely (globally) determines the allowable ordered triple of angles.  If so then this would give the strongest form of rigidity: a diffeomorphism between regions that preserves the medial axis and infinitesimally preserves the radial lines at points of the medial axis must preserve angles at branch points of the medial axis.  We conjecture that this is true.  A first step in verifying this would be to identify the subset where the rank is less than $2$ and examine the behavior of the triple ratio map at these points. 
\end{Remark}

\section{Second Order Rigidity Conditions on Diffeomorphisms Preserving the Medial Axis}  
\label{S:sec6}
\par
We have seen that at branch points there are cross ratio conditions on diffeomorphisms at branch points.  Even if these conditions are satisfied at branch points, there are also second order conditions on the diffeomorphism in terms of the radial shape operators for the two regions defined by the Blum structure.  We recall this condition to conclude our discussion.  The condition is described in full generality.  Given regions $\gW_i \subset \R^n$, $i = 1, 2$, with smooth boundaries $\cB$, resp. $\cB^{\prime}$, we suppose they have skeletal structures $(M, U)$ for $\gW_1$ and $(M^{\prime}, U^{\prime})$ for $\gW_2$.  Here $M$, resp. $M^{\prime}$, are the skeletal sets which allow relaxation of the conditions for the Blum medial axis; and $U$ resp. $U^{\prime}$ are the multivalued vector fields.  \par
We suppose that there is a diffeomorphism $\varphi$ from a neighborhood of $M_1$ to a neighborhood of $M_2$, which maps $M_1$ to $M_2$ and $d\varphi$ sends each $U(x) \mapsto U^{\prime}(\varphi(x))$.  Also, if $U = r_1\cdot U_1$ and $U^{\prime} = r_2\cdot U_1^{\prime}$ for unit vector fields $U_1$, resp. $U_1^{\prime}$, we let $U_1^{\prime}(\varphi(x)) = \gs(x) U_1(x)$ for a smooth \lq\lq scale function\rq\rq $\gs(x)$.  Each of the skeletal structures have for each smoothly varying value $U$ on a smooth point, or a singular point which is a limiting point of a smooth sheet, a radial shape operator $S_{rad} : T_x M \to T_x M$ and similarly for $M^{\prime}$.    \par 
We define a \lq\lq radial distortion operator\rq\rq\,  $Q_{\varphi} :  T_x M \to T_x M $, by 
$$  Q_{\varphi}(\bv)\,\, = \,\, - d\varphi^{-1} (\proj_{U^{\prime}}(d^2(\varphi_x(v, U_1)\, ,$$
where $\proj_{U^{\prime}}$ denotes projection along $\langle U^{\prime} \rangle$ onto $T_{\varphi(x)}M^{\prime}$.  At a point $x \in M$, let $\bv$ denote a basis for $T_xM$ with $\bv^{\prime}$ denoting the image $\bv^{\prime} = d\varphi(x)(\bv)$.  For these bases we let $S_{\bv}$, resp. $S_{\bv ^{\prime}}$, denote the matrix representations of the radial shape operators $S_{rad}$ for $(M, U)$, resp. $S_{rad}^{\prime}$ for 
$(M ^{\prime}, U ^{\prime})$.  We also let $Q_{\varphi\,\bv}$ denote the matrix representation of $Q_{\varphi}$ with respect to the basis $\bv$.  
Then there is the following relation (see \cite[Thm. 5.4]{D1}).
\begin{equation}
\label{Eqn6.1}
S_{\bv^{\prime}} \,\, = \,\, \gs(\varphi(x)) \left(S_{\bv} + Q_{\varphi \, \bv}\right)\, . 
\end{equation}
%%%  where $\bv^{\prime} = d\varphi(x)(\bv)$.  
\par
We note that in the \lq\lq partial Blum case\rq\rq for which $U$ is orthogonal to the boundary $\cB$ at the boundary point $x + U(x)$, the differential geometry of $\cB$, specifically the differential geometric shape operator is given by a specific formula in terms of $S_{rad}$ and the radial function $r$ and this formula is invertible (see e.g. \cite[\S 3]{D1}).  Thus, the relation between the differential geometry of the boundaries at each point is captured by this second order derivative information for the diffeomorphism at the corresponding medial axis point. \par
\begin{Remark}
\label{Rem6.2}
If instead $d\varphi$ only preserves the radial lines, we can replace $U^{\prime}$ by 
$\widetilde{U}^{\prime}(\varphi(x))  = d\varphi(x)(U(x))$, which gives a radial vector field on $M^{\prime}$ which has the same radial shape operators as $U^{\prime}$ (as the unit vector fields agree).  Thus, \eqref{Eqn6.1} will again hold, except $\gs$ will be replaced by the scale factor for $U$ and $\widetilde{U}^{\prime}$.   
\end{Remark}


\begin{thebibliography}{M-VI}

\bibitem[Az]{Az} Artzy, R. {\em Linear Geometry} Addison-Wesley (1965).

\bibitem[BN]{BN} Blum, H., and Nagel, R. {\em Shape description using weighted symmetric axis features} Pattern Recognition, {\bf 10}, (1978) 167--180.

\bibitem[D]{D} Damon, J. {\em Smoothness and Geometry of Boundaries Associated
to Skeletal Structures I: Sufficient Conditions for Smoothness} Ann. Inst. Fourier, {\bf 53 (6)}, (2003) 1941--1985.

\bibitem[D1]{D1} \bysame {\em Smoothness and Geometry of Boundaries Associated
to Skeletal Structures II: Geometry in the Blum Case} Compositio
Mathematica, {\bf 140(6)}, (2004) 1657--1674.

\bibitem[D2]{D2}  \bysame {\em Determining the geometry of boundaries of objects
from medial data} International Journal of Computer Vision, {\bf 63 (1)}, (2005) 45--64.

\bibitem[DG]{DG} Damon, J., and Gasparovic, E. {\em Medial/skeletal linking structures
for multi-region configurations} Memoirs of the American Mathematical
Society. {\bf vol 250 no. 1193} (2017).

\bibitem[DG2]{DG2} \bysame {\em Modeling Multi-object Configurations via Medial/Skeletal Linking Structures} International Journal of Computer Vision, {\bf 124} (2017) 255--272.

\bibitem[GK]{GK} Giblin, P. J., and Kimia,B.B. {\em A formal classification of $3D$ medial axis points and their local geometry}  IEEE Transactions on Pattern
Analysis and Machine Intelligence, {\bf 26(2)}, (2004) 238--251.

\bibitem[M]{M} Mather, J. {\em Distance from a Submanifold in Euclidean Space} 
Proc. Symp. Pure Math. {\bf Vol. 40, Pt 2}, (1983) 199--216.

\bibitem[P]{P} Pizer, S., et al. {\em Multiscale medial loci and their properties}
International Journal of Computer Vision, {\bf 55(2Ð3)} (2003), 155--179.

\bibitem[PS]{PS} Pizer, S., and Siddiqi, K. (Eds.). {\em Medial representations: Mathematics, Algorithms, and Applications} Computational imaging and
vision ({\bf Vol. 37}). Berlin, Springer (2008).

\bibitem[Yo]{Yo} Yomdin, J. {\em On the local structure of the generic central set}
Compositio Mathematica, {\bf 43} (1981) 225--238.

\bibitem[Y]{Y} Yushkevich, P., Aly, A., Wang, J., Xie, L., Gorman, R., Younes, L., Pouch, A. {\em Diffeomorphic Medial Modeling} Procs. IPMI 2019, Hong Kong, China, (2019) 208-Ð220.  

\end{thebibliography}
\end{document}